\documentclass{article} 
\usepackage{times,amsmath,amssymb,graphicx,fullpage}
\usepackage{hyperref}
\usepackage{url}
\usepackage{fouriernc}

\title{Feature Selection For High-Dimensional Clustering}

\author{Martin Azizyan \qquad Aarti Singh \qquad Larry Wasserman \\ Carnegie Mellon University}

\let\hat\widehat
\let\tilde\widetilde

\newtheorem{theorem}{Theorem}
\newtheorem{lemma}[theorem]{Lemma}

\newenvironment{proof}{{\bf Proof.}}{$\Box$}

\newenvironment{enum}{
\begin{enumerate}
  \setlength{\itemsep}{1pt}
  \setlength{\parskip}{0pt}
  \setlength{\parsep}{0pt}
}{\end{enumerate}}

\begin{document}

\maketitle

\begin{abstract}
We present a nonparametric method for selecting informative
features in high-dimensional clustering problems.
We start with a screening step that uses a test for multimodality.
Then we apply kernel density estimation and mode clustering to the selected features.
The output of the method consists of a list of relevant features, and cluster assignments.
We provide explicit bounds on the error rate of the resulting clustering.
In addition, we provide the first error bounds on mode based clustering.
\end{abstract}

\section{Introduction}

There are many methods
for feature selection in 
high-dimensional classification
and regression.
These methods require
assumptions such as
sparsity and incoherence.
Some methods (Fan and Lv 2008) also assume 
that relevant variables
are detectable through marginal correlations.
Given these assumptions,
one can prove guarantees for the performance
of the method.

A similar theory for feature selection in clustering is lacking.
There exist a number of methods
but they do not come with precise assumptions and guarantees.
In this paper we propose a method
involving two steps:
\begin{enum}
\item A screening step to eliminate uninformative features.
\item A clustering step based on estimating the modes
of the density of the relevant features.
The clusters are the basins of attraction of the modes (defined later).
\end{enum}

The screening step uses
a multimodality test such as
the dip test from
Hartigan and Hartigan (1985) or
the excess-mass test in Chan and Hall (2010).
We test the marginal distribution of each feature
to see if it is multimodal.
If not, that feature is declared to be uninformative.
The clustering is then based on mode estimation 
using the informative features.

{\em Contributions.}
We present a method for variable selection in clustering,
and an analysis of the method.
Of independent interest, we provide the first risk bounds
on the clustering error of mode-based clustering.

{\em Related Work.}
Witten-Tibshirani (2010)
propose a penalized version of $k$-means clustering,
Raftery-Dean (2006) use a mixture model with a BIC penalty,
Pan-Shen (2007)
use a mixture model with a sparsity penalty
and
Guo-Levina-Michailidis (2010)
use a pairwise fusion penalty.
None of these papers provide theoretical guarantees.
Sun-Wang-Fang (2012)
propose a $k$-means method with a penalty on the cluster means.
They do provide some consistency guarantees
but only assuming that the number of clusters $k$ is known.
Their notion of non-relevant features is different than ours;
specifically, a non-relevant feature has cluster center equal to 0.
Furthermore, their guarantees are of a different nature in that 
they show consistency of the regularized k-means objective (which is NP-hard), 
and not the 
iterative algorithm.

{\em Notation:}
We let $p$ denote a density function,
$g$ its gradient and $H$ its Hessian.
A point $x$ is a {\em local mode} of $p$ if
$||g(x)||=0$, where throughout the paper $\|\cdot\|$ denotes
the euclidean norm, and all the eigenvalues
of $H(x)$ are negative.
In general, the eigenvalues of
a symmetric matrix $A$ are denoted by
$\lambda_1\geq \lambda_2\geq \cdots$.
We write $a_n \preceq b_n$ to mean that there is some $C>0$
such that $a_n \leq C b_n$ for all large $n$. $C,c$ will denote different constants. 
We use $B(x,\epsilon)$ to denote a closed ball of radius
$\epsilon$ centered at $x$.

\section{Mode Clustering}

Here we give a brief review of mode clustering, also called mean-shift clustering;
more details can be found in Cheng (1995), Comaniciu and Meer (2002),
Arias-Castro, Mason, Pelletier (2014) and Chacon (2012).

Let $X_1,\ldots, X_n \in \mathbb{R}^d$
be random vectors
drawn from a distribution $P$
with density $p$.
We write
$X_i = (X_i(1),\ldots, X_i(d))^T$
to denote the $d$ features of observation $X_i$.
We assume that $p$ has a finite set of modes 
${\cal M}=\{m_1,\ldots, m_k\}$.
The population clustering associated with $p$ is
${\cal C} = \{ {\cal C}_1,\ldots, {\cal C}_k\}$
where ${\cal C}_j$ is the basin of attraction
of $m_j$.
That is, $x\in {\cal C}_j$ if the gradient ascent curve, or flow, 
starting at $x$ ends at $m_j$.
More precisely, the {\em flow}
starting at $x$ is the path
$\pi_x: \mathbb{R}\to \mathbb{R}^d$
satisfying $\pi_x(0) =x$ and
$\pi_x'(t) = \nabla p(\pi_x(t))$.
Then $x\in {\cal C}_j$ iff
$\lim_{t\to\infty}\pi_x(t)=m_j$.
Let $m(x)\in {\cal M}$ denote the mode to which $x$ is assigned.
Thus $m: \mathbb{R}^d \to {\cal M}$.
Define the clustering function
$c:\mathbb{R}^d \times \mathbb{R}^d\to \{0,1\}$ by
$$
c(x,y) =
\begin{cases}
1 & {\rm if\ } m(x) = m(y)\\
0 & {\rm if\ } m(x) \neq m(y).
\end{cases}
$$
Thus, $c(x,y)=1$ if and only if $x$ and $y$ are in the same cluster.

Now let $\hat p$ be an estimate of the density $p$
with corresponding estimated modes
$\hat{\cal M}= \{\hat m_1,\ldots, \hat m_\ell\}$,
mode assignment function $\hat m$,
and basins
$\hat{\cal C} = \{ \hat {\cal C}_1,\ldots, \hat {\cal C}_\ell\}$.
(The modes and cluster assignments can be found numerically
using the mean shift algorithm; see Cheng (1995) and
Comaniciu and Meer (2002).)
This defines a sample cluster function $\hat c$.
The clustering loss is defined to be
\begin{equation}
L = \frac{1}{\binom{n}{2}}\sum_{j < k} I \Bigl( \hat c(X_j,X_k)\neq c(X_j,X_k)\Bigr).
\end{equation}
A second loss function is the Hausdorff distance 
$H(\hat{\cal M},{\cal M})$
where
$$
H(C,D) = \inf \{\epsilon: C \subset D\oplus \epsilon\ {\rm and}\ 
D \subset C\oplus \epsilon\}
$$
and
$A\oplus\epsilon =\cup_{x\in A}B(x,\epsilon)$.

\section{The Method}

Now we describe the steps of our algorithm.

\rule{6in}{1mm}

\noindent
\begin{enumerate}
\item 
(Screening)
Let $p_j$ be the marginal density of the $j^{\rm th}$ feature.
Let $k_j$ be the number of modes of $p_j$.
We test
$$
H_0: k_j \leq 1 \ \ \ \ \ {\rm versus}\ \ \ H_1: k_j >1.
$$
The test is given in Figure \ref{fig::test}.
Let $R = \{j:\ H_0\ {\rm was\ rejected}\}$ and let
$r =|R|$.
\item 
(Mode Clustering)
Let
$Y_i = (X_i(a):\ a\in R)$
be the relevant coordinates of $X_i$.
Estimate the density of $Y$
with the kernel density estimator
$$
\hat p_h(y) = \frac{1}{n}\sum_{i=1}^n \frac{1}{h^r}
K\left(\frac{||Y_i-y||}{h}\right)
$$
with bandwidth $h$.
Let $\hat {\cal M}$ be the modes
corresponding
to $\hat p_h$
with corresponding basins
$\hat{\cal C} = \{ \hat {\cal C}_1,\ldots, \hat {\cal C}_\ell\}$
and cluster function
$\hat c$.
\item Output 
$\hat {\cal M}$, $R$,
$\hat m(Y_1),\ldots, \hat m(Y_n)$ and $\hat c$.
\end{enumerate}

\rule{6in}{1mm}


\begin{figure}
\centering
\fbox{\parbox{5.3in}{
\begin{center}
{\sf Test For Multi-Modality}
\end{center}
\begin{enumerate}
\item Fix $0 < \alpha < 1$. Let $\tilde\alpha = \alpha /(nd)$.
\item For each $1\leq j \leq d$,
compute $T_j = {\rm Dip}(F_{nj})$ where
$F_{nj}$ is the empirical distribution function of the $j^{\rm th}$ feature
and ${\rm Dip}(F)$ is defined in (\ref{eq::dip}).
\item Reject the null hypothesis that feature $j$ is not multimodal if
$T_j > c_{n,\tilde\alpha}$ where
$c_{n,\tilde\alpha}$ is the critical value for the dip test.
\end{enumerate}
}}
\caption{The multimodality test for the screening step.}
\label{fig::test}
\end{figure}

\subsection{The Multimodality Test}

Any test of multimodality may be used.
Here we describe the {\em dip test}
(Hartigan and Hartigan, 1985).
Let
$Z_1,\ldots, Z_n \in [0,1]$
be a sample from a distribution $F$.
We want to test
``$H_0: F$ is unimodal''
versus
``$H_1: F$ is not unimodal.''
Let ${\cal U}$ be the set of unimodal
distributions.
Hartigan and Hartigan (1985) define
\begin{equation}\label{eq::dip}
{\rm Dip}(F) = \inf_{G\in {\cal U}}\sup_x | F(x) - G(x)|.
\end{equation}
If $F$ has a density $p$ we also write
${\rm Dip}(F)$ as
${\rm Dip}(p)$.
Let $F_n$ be the empirical distribution function.
The dip statistic is
$T_n = {\rm Dip}(F_n)$.
The dip test rejects $H_0$ if
$T_n > c_{n,\alpha}$
where the critical value
$c_{n,\alpha}$ is 
chosen so that, under $H_0$,
$\mathbb{P}(T_n > c_{n,\alpha}) \leq \alpha$.\footnote{
Specifically, 
$c_{n,\alpha}$ can be defined by
$\sup_{G\in {\cal U}} P_G(T_n > c_{n,\alpha}) = \alpha$.
In practice,
$c_{n,\alpha}$ can be defined by
$P_U(T_n > c_{n,\alpha}) = \alpha$ where 
$U$ is Unif(0,1).
Hartigan and Hartigan (1985) suggest that this suffices asymptotically.}

Since we are conducting multiple tests,
we cannot test at a fixed error rate $\alpha$.
Instead, we replace $\alpha$ with
$\tilde\alpha = \alpha/(nd)$.
That is, we test each marginal and
we reject $H_0$ if
$T_n > c_{n,\tilde\alpha}$.
By the union bound,
the chance of at least one false rejection of $H_0$ is at most
$d\tilde\alpha = \alpha/n$.

There are more refined tests such as the excess mass test
given in Chan and Hall (2010),
building on work by Muller and Sawitzki (1991).
For simplicity, we use the dip test in this paper;
a fast implementation of the test is available in R.

\subsection{Bandwidth Selection}

Bandwidth selection for kernel density estimation is an enormous topic.
A full investigation of bandwidth selection in mode clustering
is beyond the scope of this paper but here we provide some
general guidance.
We may want to choose a bandwidth that gives accurate
estimates of the gradient of the density.
Based on Wand, Duong, and Chacon (2011) this suggests
$h_n = S \left(\frac{4}{r+4}\right)^{1/(6+r)} n^{-1/(6+r)}$
where $S$ is the average of the sample standard deviations along each coordinate.
On the other hand,
in the low noise case (well-separated clusters)
we may want to choose an $h>0$
that does not go to 0 as $n$ increases.
Inspired by similar ideas used in RKHS methods 
(Sriperumbudur et al. 2009)
one possibility is to take $h$ to be the 0.05 quantile
of the values $||Y_i - Y_j||$.
Finally, we note that Einbeck (2011)
has a heuristic method for choosing $h$ for mode clustering.

\section{Theory}

\subsection{Assumptions}

We make the following assumptions:

\noindent
{\bf (A1) (Smoothness)} $p$ has three bounded, continuous derivatives.
Thus, $p\in C^3$.
Also, $p$ is supported on a compact set
which we take to be a subset of
$[0,1]^d$.

\noindent
{\bf (A2) (Modes)} $p(y)$ has finitely many modes
${\cal M} = \{m_1,\ldots, m_k\}$
where $y\in\mathbb{R}^s$ is the subset of $x$ defined in (A3).
Furthermore, $p$ is a Morse function, i.e.
the Hessian at each critical point is non-degenerate.
Also, there exists $a>0$ such that
$\min_{j\neq \ell}||m_j - m_\ell|| \geq a$.
Finally,
there exits $0 < b < B < \infty$ 
and $\gamma>0$
such that,
\begin{equation}
-B \leq  \min_j \lambda_{s}(H(y)) \leq
\max_j \lambda_{1}(H(y)) \leq -b
\end{equation}
for all $y\in B(m_j,\gamma)$ and 
$1 \leq j \leq k$.

\noindent
{\bf (A3) (Sparsity)} The true cluster function $c$
depends only on a subset of features 
$S\subset\{1,\ldots, d\}$ of size $s$.
Let $y=(x(i):\ i\in s)$ denote the relevant features.

\noindent
{\bf (A4) (Marginal Signature)}
If $j\in S$, then the marginal density $p_j$
is multimodal.
In particular, 
\begin{equation}
\min_{j\in S}{\rm Dip}(p_j) > \sqrt{ \frac{2c_n\log(2nd)}{n}}
\end{equation}
where $c_n$ is any slowly increasing function of $n$
(such as $\log n$ or $\log \log n$)
and
\begin{equation}
{\rm Dip}(p) = \inf_{q\in {\cal U}}\sup_x |F_p(x) - F_q(x)|
\end{equation}
where $F_p(x) = \int_{-\infty}^x p(u) du$
and ${\cal U}$ is the set of unimodal distributions.

\noindent
{\bf (A5) (Cluster Boundary Condition)}
Define the {\em cluster margin}
\begin{equation}\label{eq::boundary}
\Omega_\delta = \Biggl(\bigcup_{j=1}^k (\partial {\cal C}_j)\Biggr)\oplus \delta
\end{equation}
where
$\partial {\cal C}_j$ is the boundary of ${\cal C}_j$
and $A\oplus \delta = \bigcup_{y\in A}B(y,\delta)$.
We assume that there exists $c>0$ and $\beta\geq 1$ such that,
 for all small $\delta>0$,
$P(\Omega_\delta)\leq c \delta^\beta.$

\subsection{Discussion of the Assumptions}

Assumption (A1) is a standard smoothness assumption.
Assumption (A2) is needed to make sure that
the modes are well-defined and estimable.
Similar assumptions appear in 
Arias-Castro, Mason, Pelletier (2014)
and Romano (1988), for example.
Assumption (A3) is needed in the high-dimensional setting
just as in high-dimensional regression.

Assumption (A4) is the most restrictive assumption.
The assumption is violated when clusters are very close together
and are not well-aligned with the axes.
To elucidate this assumption, consider Figure \ref{fig::marginal1}.
The left plots show a violation of (A4).
The middle and right plots show cases where the assumption holds.
It may be possible to relax (A4) but, as far as we know,
every variable selection method for clustering in high dimensions
makes a similar assumption
(although it is not always made explicit).

Assumption (A5) is satisfied with $\beta=1$
for any bounded density with cluster  
boundaries are not space-filling curves.
The case $\beta >1$ corresponds to well-separated clusters.
This implies that there is not too much mass at the cluster boundaries.
This can be thought of as a cluster version of
Tsybakov's low noise assumption in classification
(Audibert and Tsybakov, 2007).
In particular,
the very well-separated case, where 
there is no mass
right on the cluster boundaries,
corresponds to $\beta = \infty$.

\begin{figure}
\begin{center}
\includegraphics[scale=.3]{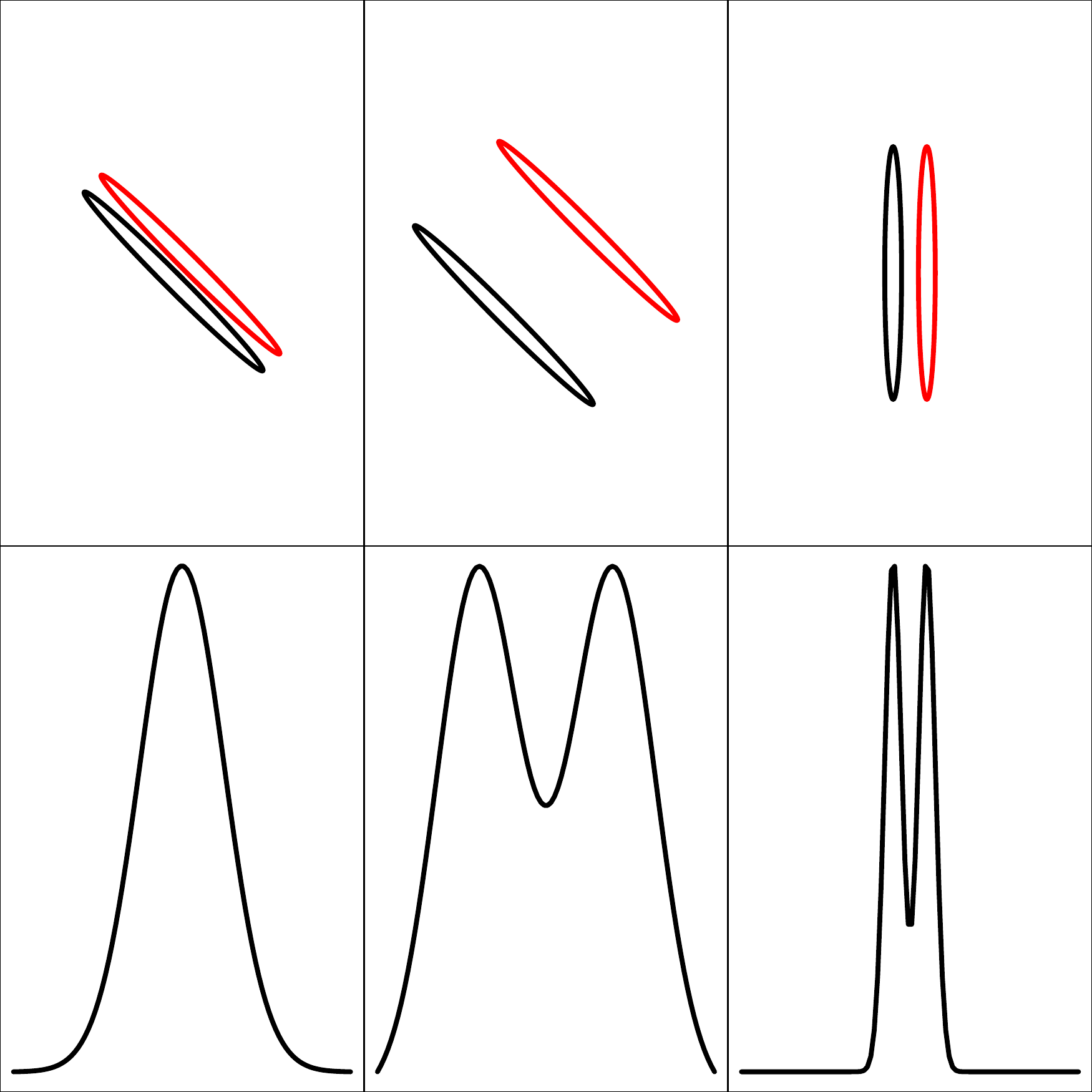}
\end{center}
\caption{\em
Three examples, each showing two clusters and two features $X(1)$ and $X(2)$.
The top plots show the clusters.
The bottom plots show the marginal density of $X(1)$.
Left: The marginal fails to reveal any clustering structure.
This example violates the marginal signature assumption.
Middle:
The marginal is multimodal and hence correctly identifies $X(1)$ as a relevant feature.
This example satisfies the marginal signature assumption.
Right:
In this case, $X(1)$ is relevant but $X(2)$ is not.
Despite the fact that the clusters are close together,
the marginal is multimodal and hence correctly identifies $X(1)$ as a relevant feature.
This example satisfies the marginal signature assumption.}
\label{fig::marginal1}
\end{figure}

\subsection{Main Result}

\begin{theorem}
\label{thm::main}
Assume (A1)-(A5).
Then
$\mathbb{P}(R=S) > 1 - 2/n$.
Furthermore, we have the following:
\begin{enumerate}
\item
Let
$\eta_j = \sup_x||\hat p_h^{(j)}(x) - p^{(j)}(x)||$
where $p^{(j)}$ denotes the $j^{th}$ derivative of the density.
The cluster loss is bounded by
\begin{equation}
\mathbb{E}[L]  \leq
e^{-nc h^{s+4}} +
\left(\frac{C_1}{\log\left(\frac{C_2}{\eta_1}\right)}\right)^\beta + \frac{2}{n} 
\end{equation}
where
$\eta_1 \preceq h^2 + \sqrt{\frac{\log n}{n h^{s+2}}}$.
Choosing $h_n \asymp n^{-b}$ for
$0 < b < 1/(4+s)$, we have
\begin{equation}
\mathbb{E}[L] \preceq
e^{-c n^{\omega}} +
\left(\frac{1}{\log n}\right)^\beta
\end{equation}
where $\omega = 1 - b(s+4) > 0$ and $\beta$ is a constant.
\item
(Low noise and fixed bandwidth.)
Suppose 
that
$\beta \succeq \frac{2 v \log n}{\log\log (1/h^2)}.$
If $0 < h < Ca$ then
$\mathbb{E}[L]\preceq n^{-v} + c_1 e^{-n c}.$
\item Except on a set of probability at most $O(e^{-nc h^s})$,
$$
H(\hat{\cal M},{\cal M}) \preceq \sqrt{ h^2 + \sqrt{\frac{\log n}{n h^s}}}.
$$
Hence, if $h>0$ is fixed but small, the for any $K$ and large enough $n$,
$H(\hat{\cal M},{\cal M}) < \min_{j\neq k}||m_j-m_k||/K$.
If $h\asymp n^{-1/(4+s)}$, then
$H(\hat{\cal M},{\cal M}) = O( (\log n)^{1/2}/n^{\frac{1}{4+s}})$.
\end{enumerate}
\end{theorem}

The first result shows that
the clustering error depends on the number of relevant variables $s$
and on the boundary exponent $\beta$.
The second result shows that in the low noise (large $\beta$) case,
we can use a small but non-vanishing bandwidth.
In that case, the clustering error for all pairs of points
not near the boundary is exponentially small and
the fraction of points near the boundary decreases as a polynomial in $n$.
The third result shows that
the Hausdorff distance between the estimated modes and true modes
is small relative to the mode separation
with high probability even if $h$ does not tend to 0.
When $h$ does tend to 0,
the Hausdorff distance shrinks at rate
$O( (\log n)^{1/2}/n^{\frac{1}{4+s}})$.

\section{Proofs}
\vspace{-0.1in}
\subsection{Screening}

\begin{lemma}
[False Negative Rate of the dip test.]
Let $T_n$ be the dip statistic.
Let $\delta = {\rm Dip}(p)$.
Suppose that $\sqrt{n}\delta \to \infty$.
Then
$\mathbb{P}(T_n \leq c_{n,\alpha}) < 2 e^{- n\delta^2/2}$.
\end{lemma}

\begin{proof}
It follows from Theorem 3 of Hartigan and Hartigan (1985)
that $c_{n,\alpha} \sim C/\sqrt{n}$ for some $C>0$.
Since $\sqrt{n}\delta \to \infty$,
we have that the event $\{T_n \leq c_{n,\alpha}\}$
implies the event
$\{T_n \leq \delta/2\}$.
Let $F_0$ be the member of ${\cal U}$ closest to $F$ and let
$\hat F_0$ be the member of ${\cal U}$ closest to $F_n$.
Then
$T_n \leq \delta/2$ implies that
$$
\delta < \sup_x |F(x) - F_0(x)| \leq
\sup_x |F(x) - \hat F_0(x)| \leq
\sup_x |F(x) - F_n| + 
\sup_x |F_n - \hat F_0(x)|  \leq
\sup_x |F(x) - F_n| + \frac{\delta}{2}
$$
and so
$\sup_x |F(x) - F_n| > \delta/2$.
In summary,
the event
$\{T_n \leq c_{n,\alpha}\}$ implies the event
$\{\sup_x |F(x) - F_n| > \delta/2\}$.
According to the Dvoretsky-Kiefer-Wolfowitz theorem,
$\mathbb{P}( \sup_x |F(x) - F_n| > \epsilon)\leq 2 e^{-2n\epsilon^2}.$
Hence,
$\mathbb{P}(T_n \leq c_{n,\alpha}) \leq
\mathbb{P}( \sup_x |F(x) - F_n| > \delta/2)\leq 2 e^{-n \delta^2/2}.$
\end{proof}

\begin{lemma}
[False negative rate: Multiple Testing Version.]
Recall that $\tilde\alpha = \alpha/(nd)$.
Let $T_n$ be the dip statistic.
Let $\delta = {\rm Dip}(p)$.
Suppose that $\sqrt{n/\log(nd)}\delta \to \infty$.
Then
$\mathbb{P}(T_n \leq c_{n,\tilde\alpha}) < 2 e^{- n\delta^2/2}$.
\end{lemma}

{\bf Proof Outline.}
As noted in the proof of the previous lemma,
it follows from Theorem 3 of Hartigan and Hartigan (1985)
that for fixed $\alpha$,
$c_{n,\alpha} \sim C/\sqrt{n}$ for some $C>0$.
The proof uses that fact that
$\sup_{0\leq x\leq 1} |\sqrt{n}(F_n(x) - x) - B(x)|\to 0$
in probability,
where $B$ is a Brownian bridge.
A simple extension, using the properties of a Brownian bridge,
shows that
$c_{n,\tilde\alpha} \sim \sqrt{\log (nd)/n}$.
The rest of the proof is then the same as the previous proof. $\Box$

\begin{lemma}
[Screening Property]
Recall that $R$ is the set of $j$ not rejected by the dip test.
Assume that
$$
\min_{j\in S}{\rm Dip}(p_j) > \sqrt{ \frac{2c_n }{n} \log(2nd)}.
$$
Then, for $n$ large enough,
$\mathbb{P}(R=S) > 1 - \frac{2}{n}.$
\end{lemma}

\begin{proof}
By the union bound and the previous lemma,
the probability of omitting any $j\in S$ is at most
$2s e^{-n\delta^2/2} < 1/n$ where
$\delta = \min_{j\in S}{\rm Dip}(p_j)$.
On the other hand,
probability of including any feature $j\in S^c$
is at most $\tilde\alpha = d\alpha/(nd) = \alpha/n < 1/n$.
\end{proof}

\subsection{Mode and Cluster Stability}

Now we need some properties of density modes.
Recall that $p\in C^3$, has $k$ modes
$m_1,\ldots, m_k$ separated by $a>0$
and by (A2),
the Hessian $H(m)$ at each mode $m$
has eigenvalues 
in $[-B, -b]$ for some
$0 < b < B < \infty$.
Let
$\kappa_j = \sup_x ||p^{(j)}||$.
Since $p\in C^3$,
$\kappa_j$ is finite for
$j=0,1,2,3$.
Let $\tilde p\in C^3$ be another density.
Let
$\eta_j = \sup_x ||p^{(j)} - \tilde p^{(j)}||.$
Later, $\tilde p$ will be taken to be an estimate of $p$.
For now, it is just another density that is close to $p$.
We want to show that $\tilde p$ has similar clusters to $p$.

{\bf (A6)
Assume that
$\eta_0 < a^2/8$,
$\eta_0 < 9/(128 \kappa_3)$ and 
$\eta_2 < b/2$.}

\begin{lemma}
\label{lemma::prop1}
Assume (A1) - (A6).
Then
$\tilde p$ has exactly $k$ modes
$\tilde m_1,\ldots, \tilde m_k$.
After an appropriate relabeling of the indices, we have
$\max_{1\leq j\leq k} ||m_j - \tilde m_j|| \leq  \sqrt{8\eta_0}.$
\end{lemma}

The proof is in the supplementary material.

\begin{lemma}
\label{lemma::rightmode}
Suppose that $m(x) = m_j$.
Let
$\delta = \frac{C_1}{\log\left(\frac{C_2}{\eta_1}\right)}.$
Let
$d(x) = \inf\Biggl\{||x-y||:\ y\in \bigcup_j \partial C_j\Biggr\}$
be the distance of $x$ from the cluster boundaries.
If $\sqrt{\eta_0} < C_3 a$
and if
$d(x) > \delta$,
then
$\tilde m(x) = \tilde m_j$.
\end{lemma}

\begin{proof}
There are two cases:
$x\in B(m_j,\sqrt{\epsilon})$ and
$x\notin B(m_j,\sqrt{\epsilon})$.
The more difficult case is the latter;
we omit the first case.
As $x$ is not on the boundary
and not in
$B(x,\sqrt{\epsilon})$, we have that
$||g(x)|| \neq 0$
and in particular,
$||g(x) || \geq \frac{C_4}{\log\left(\frac{C_2}{\eta_1}\right)}.$
Fix a small $\epsilon >0$.
There exists $t_\epsilon$,
depending on $x$, such that
$\pi_x(t_\epsilon)\in B(m_j,C_2\sqrt{\epsilon})$.
From Lemma \ref{lemma::boundary} below, we have
$$
t_\epsilon \leq
\frac{C_5}{||g(x_0)||} +
\frac{\frac{1}{2}\log(1/\epsilon) + \log ||x_0-m||}{b}.
$$
From this, it follows that
$\epsilon + 2\eta_0 +  \frac{\kappa_1}{\sqrt{d}\kappa_2}\eta_1 e^{\sqrt{d}\kappa_2 t_\epsilon} < C_6$
for $C_6 < \infty$.
This equation implies,
from the proof of Theorem 2 of Arias-Castro et al,
that
$||\lim_{t\to\infty}\tilde\pi_x(t) - m_j||\leq C_4 \sqrt{\eta_0}$.
Since $C_4 \sqrt{\eta_0} < a$, when $\eta_0$ is small enough
we conclude that
$\lim_{t\to\infty}\tilde\pi_x(t) = \tilde m_j$.
\end{proof}

\begin{lemma}
\label{lemma::boundary}
Consider the flow $\pi$ starting at a point $x_0$
and ending at a mode $m$.
For some $C_6>0$,
$$
t_\epsilon \leq
\frac{C_6}{||g(x_0)||} +
\frac{\frac{1}{2}\log(1/\epsilon) + \log ||x_0-m||}{b}.
$$
\end{lemma}

The proof is in the supplementary material.

The next lemma shows that if $x$ and $y$ are in the same cluster
and not too close to a cluster boundary,
then $x$ and $y$ are also in the same cluster relative to $\tilde p$.

\begin{lemma}
\label{lemma::prop2}
Suppose that (A1)-(A6) holds and that
$\sqrt{\eta_0} < C_4 a$.
Suppose that
$x,y \in {\cal C}_j$ and hence
$m(x) = m(y)=m_j$ and
$c(x,y) =1$.
Furthermore,
suppose that
$x,y \notin \Omega_\delta$.
(Recall that $\Omega_\delta$ is defined in (\ref{eq::boundary}).)
Then
$\tilde m(x) = \tilde m(y) = \tilde m_j$ and so
$\tilde c(x,y) =1$.
\end{lemma}

\begin{proof}
Since $x,y \notin \Omega_\delta$,
from the definition of $\delta$
and
from Lemma
\ref{lemma::rightmode}
it follows
that
$\lim_{t\to\infty}\tilde \pi_x(t) = \tilde m_j$
and
$\lim_{t\to\infty}\tilde \pi_y(t) = \tilde m_j$.
\end{proof}

Next we show that if $x$ and $y$ are in different clusters
and not too close to a cluster boundary,
then $x$ and $y$ are in different clusters under $\tilde p$.
The proof is basically the same as the last proof and so is omitted.

\begin{lemma}
\label{lemma::prop3}
Assume that same conditions as in the previous lemma.
Suppose that
$m(x) = m_j$,
$m(y) = m_s$ with $s\neq j$.
Hence, $c(x,y) =0$.
Furthermore,
suppose that
$x,y \notin \Omega_\delta$.
Then
$\tilde m(x) = \tilde m_j$,
$\tilde m(x) = \tilde m_s$,
and
$\tilde c(x,y) =0$.
\end{lemma}

\subsection{Proof of Main Theorem}

We have already shown that $R=S$
except on a set of probability at most $2/n$.
Assume in the remainder of the proof that $R=S$.

Now
$\mathbb{E}[L] = (\binom{n}{2})^{-1}\sum_{j<k} \mathbb{E}[I_{jk}]$
where
$I_{jk}= I \Bigl( \hat c(X_j,X_k)\neq c(X_j,X_k)\Bigr).$
Let
$\delta = C_1/\log(C_2/\eta_1)$.
Then
$$
\mathbb{E}[I_{jk}] \leq
\mathbb{E}[I_{jk} I( (X_j,X_k)\in \Omega_\delta^c)] + \mathbb{P} ((X_j,X_k)\notin \Omega_\delta^c).
$$
Consider
$(X_j,X_k)\in \Omega_\delta^c$;
then
$I_{jk}=0$ if
$\hat p_h$ satisfies (A6) and the condition of Lemma
\ref{lemma::rightmode}.
In other words,
$I_{jk}=0$ if,
$\sqrt{\eta_0} < C_8 a$ and 
$\eta_2 < b/2$ where
$\eta_0 = \sup_x||\hat p_h(x) - p(x)||$ and
$\eta_2 = \sup_x||\hat p_h^{(2)}(x) - p^{(2)}(x)||$.
Let $p_h$ be the mean of $\hat p_h$. 
Then
$\eta_0 \leq  \sup_x|| p_h(x) - p(x)|| + \sup_x|| \hat p_h(x) - p_h(x)||$.
The first term is $O(h^2)$ which is less than
$C_8^2a^2/2$ for small $h$.
By standard concentration of measure results,
$$
\mathbb{P}(\sup_x|| \hat p_h(x) - p_h(x)|| > \epsilon) \preceq
e^{-nc h^s \epsilon^2}
$$
where $c>0$ is a constant whose value may change in different expressions.
So
$$
\mathbb{P}(\sqrt{\eta_0} > C_8 a)\leq
\mathbb{P}(\sup_x|| \hat p_h(x) - p_h(x)|| > C_8^2 a^2/2) \leq
e^{-nc h^s}.
$$
A similar analysis for $\eta_2$ yields
$\mathbb{P}(\sqrt{\eta_0} > b/2) \leq
e^{-nc h^{s+4} b^2/4}.$
Therefore,
$\mathbb{E}[I_{jk} I( (X_j,X_k)\in \Omega_\delta^c)] \preceq e^{-nc h^{s+4}}.$
Now
$\mathbb{P} ((X_j,X_k)\notin \Omega_\delta^c) \preceq
P(\Omega_\delta) \preceq \delta^\beta.$
With high probability,
$$
\eta_1 = 
O\left( h^2 + \sqrt{\frac{\log n}{n h^{s+2}}}\right).
$$
Hence, if $h=n^{-b}$,
$\delta^\beta \preceq (1/\log n)^\beta$.

The second statement follows from the first
by inserting a small fixed $h>0$ and
noting that
the fraction of points
near the boundary is 
$\theta_n = O_P(\delta^\beta) = O_P(1/n^v)$
due to the condition on $\beta$. 

For the third statement,
note that once $\eta_0$ is small enough,
the previous results imply that
${\cal M}$ and 
$\hat{\cal M}$ 
have the same cardinality.
In this case, the Hausdorff distance is, after relabelling the indices,
$H(\hat{\cal M},{\cal M}) = \max_j ||\hat m_j - m_j||$.
Once $\eta_0$ is small enough,
Lemma \ref{lemma::prop1} implies
$\max_j ||\hat m_j - m_j|| \leq \sqrt{8 \eta_0}$.
The result follows from the bounds on $\eta_0$ above.
$\Box$

\section{Example}
\vspace{-0.1in}
\begin{figure}
\begin{center}
\begin{tabular}{ccc}
\includegraphics[scale=.2]{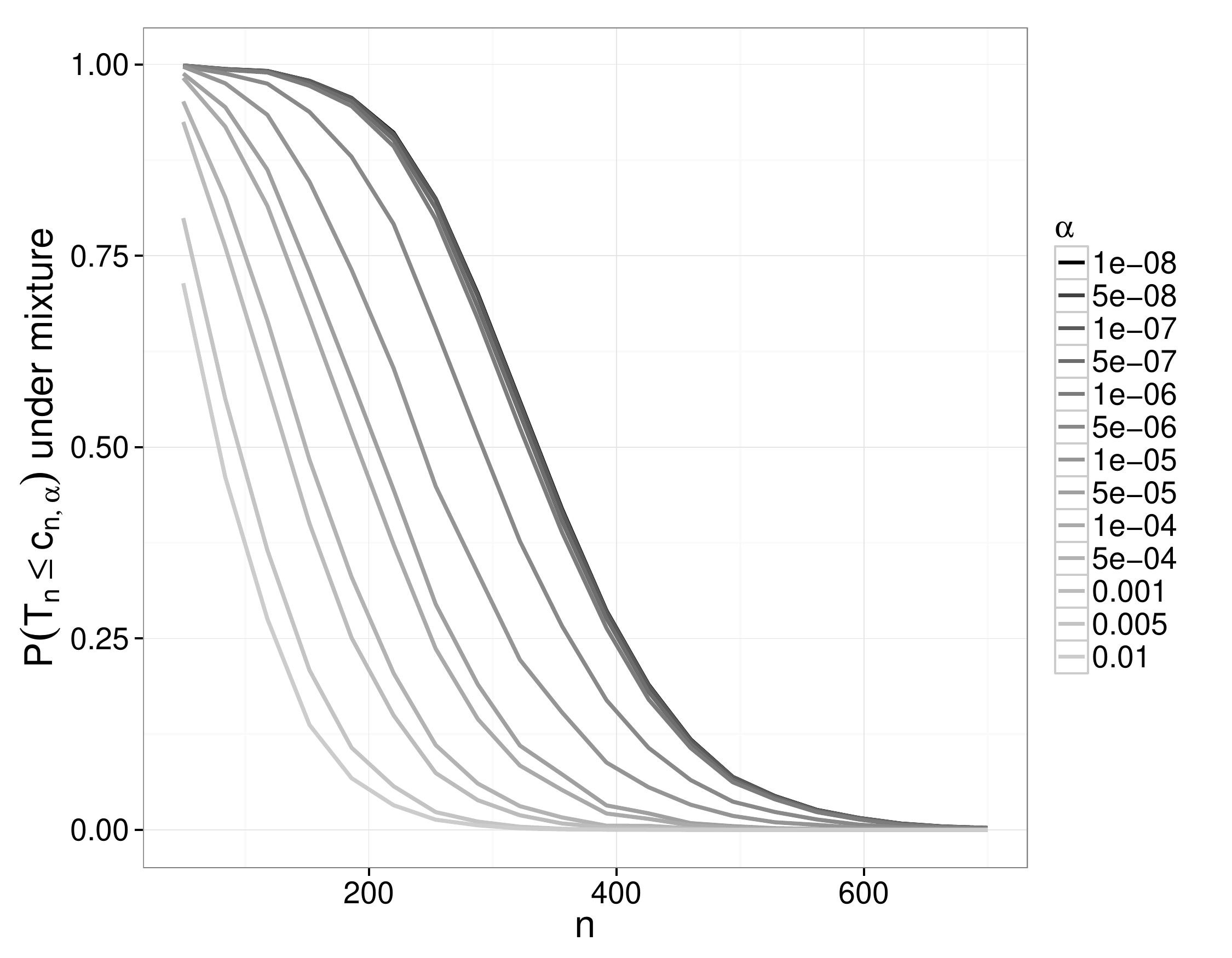} &
\includegraphics[scale=.2]{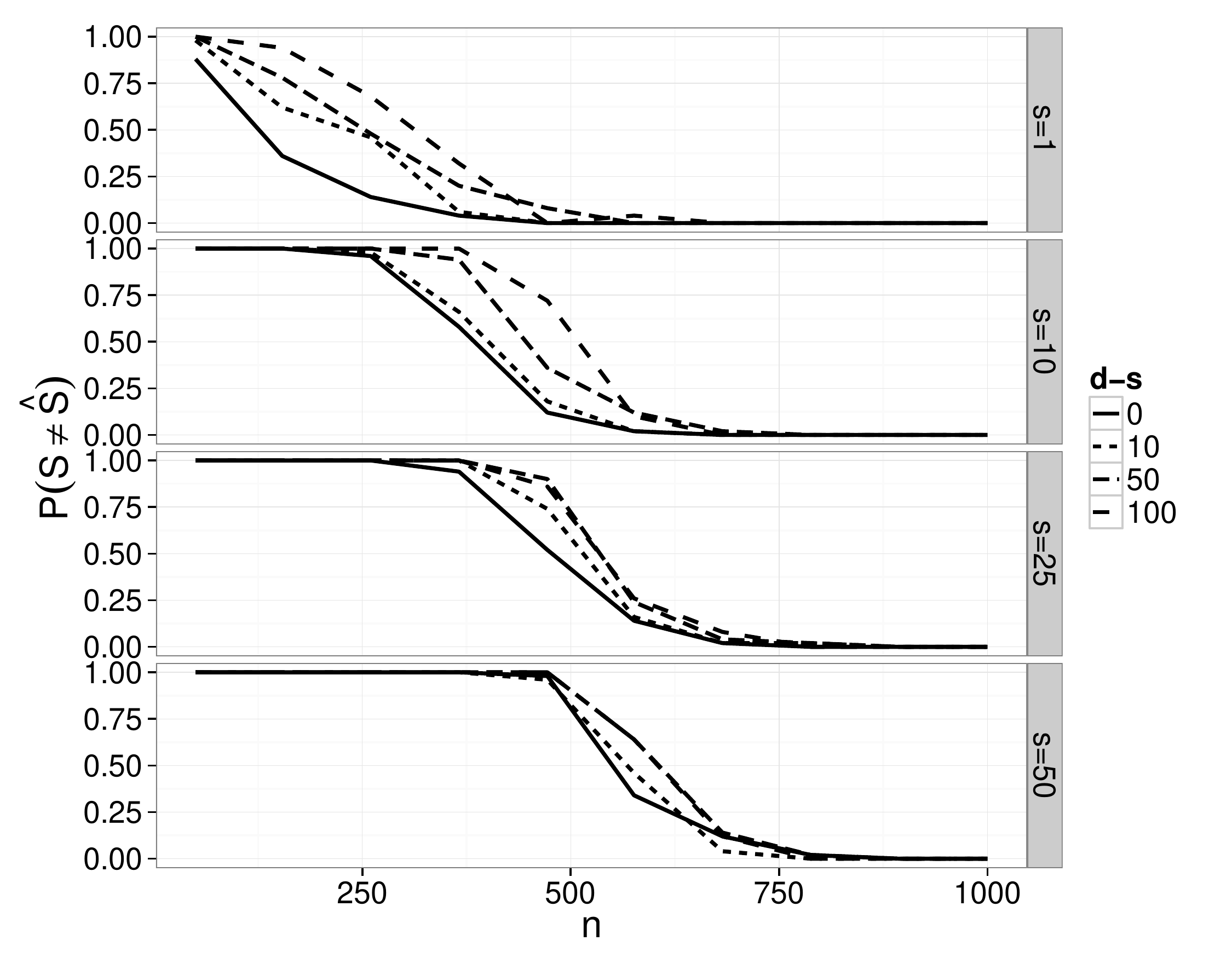} &
\includegraphics[scale=.2]{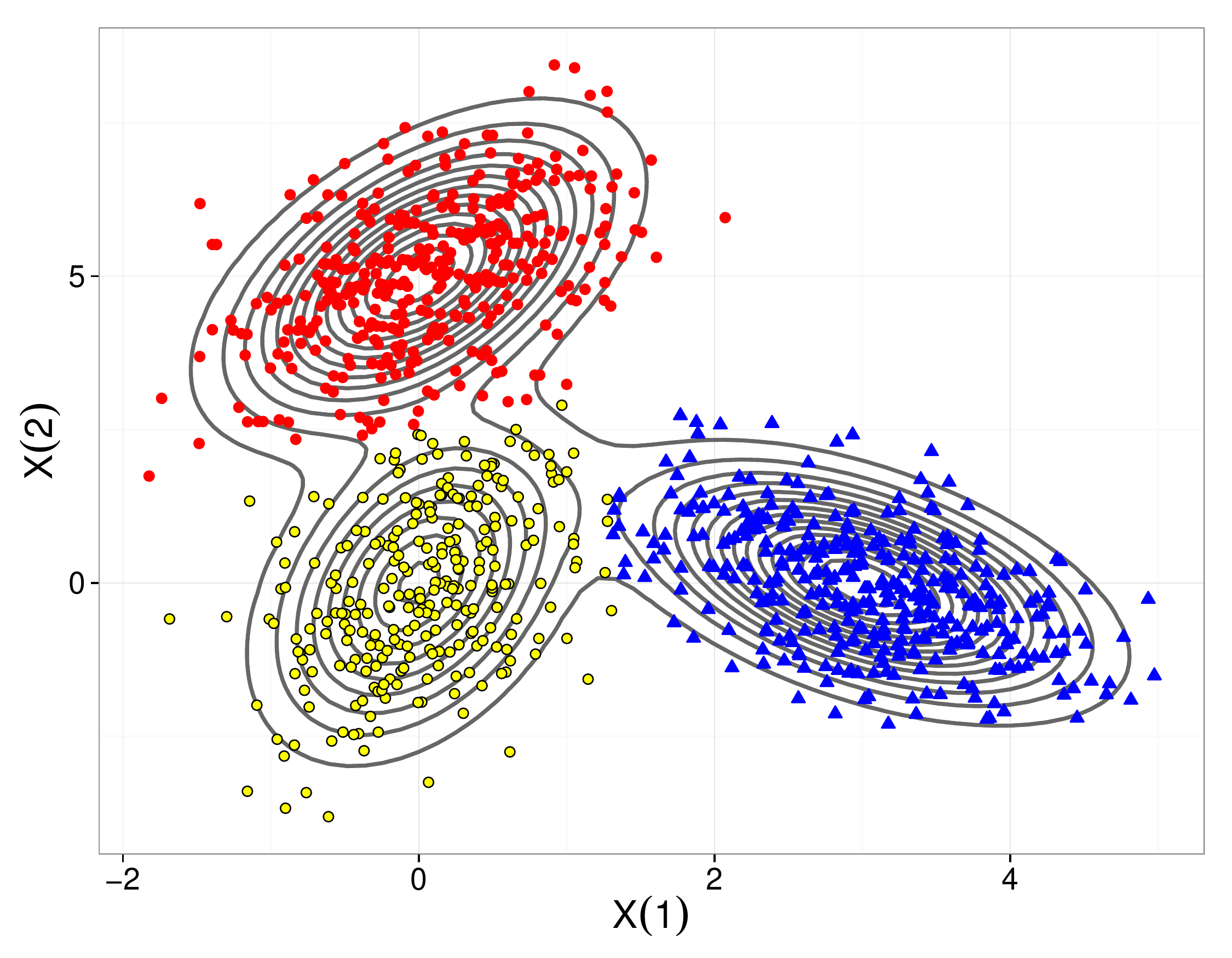}\vspace{-0.1in}
\end{tabular}
\end{center}
\caption{Left: false negative rate as a function of $\alpha$.
Middle: overall screening error rate.
Right: Final clustering based on relevant features.}\vspace{-0.1in}
\label{fig::example}
\end{figure}



In this section we give a brief
example of the proposed method.
First, we show the type II error (false negative rate)
of the dip test as a function of $\alpha$.  We use a version of the test implemented in the R package
{\em diptest}. 
We take
$P=\frac{1}{2}\mathcal{N}(0,1)+\frac{1}{2}\mathcal{N}(4,1)$.  For a
range of values for $n$, we draw $n$ samples from the mixture $10000$
times.
The left plot in Figure \ref{fig::example} shows the fraction of
times the dip test failed to detect multimodality at the specified
values for $\alpha$.  The increase in the sample size required for a
certain power appears to be at most logarithmic in $1/\alpha$.


We show the overall error rate of the support estimation procedure
in the middle plot in Figure \ref{fig::example} for the following multivariate
distribution.  For given values of $d$ and $s$, we use the Gaussian
mixture
$\frac{1}{2}\mathcal{N}(0,I)+\frac{1}{2}\mathcal{N}(4\mu_{s,d},I)$,
where $\mu_{s,d}\in\mathbb{R}^d$ contains $s$ ones followed by $d-s$
zeroes, so that the true support is $S=\{1,\ldots,s\}$.  
The plot shows the fraction of times the estimated
support $\widehat{S}$ did {\em not} exactly recover $S$ in $50$
replications of the experiment for each combination of parameters.  We
set $\alpha=0.1$ (and $\widetilde{\alpha}=\alpha/(nd)$).  All the
errors were due to incorrectly removing one of the multimodal
dimensions -- in other words, in every single instance it was the case
that $\widehat{S}\subseteq S$.  
This is not surprising since the dip test can be conservative.


Finally, we apply the full method to a $d=20$ dimensional data set
distributed in the first two dimensions according to the Gaussian
mixture
$$
\frac{2}{8}\mathcal{N}\left(\left(\begin{array}{c}0\\0\end{array}\right), 
\left(\begin{array}{cc}0.3&0.3\\0.3&2\end{array}\right)\right) + 
\frac{3}{8}\mathcal{N}\left(\left(\begin{array}{c}3\\0\end{array}\right), 
\left(\begin{array}{cc}0.6&-0.4\\-0.4&1\end{array}\right)\right) + 
\frac{3}{8}\mathcal{N}\left(\left(\begin{array}{c}0\\5\end{array}\right), 
\left(\begin{array}{cc}0.45&0.45\\0.45&1.6\end{array}\right)\right),
$$ 
and according to independent standard Gaussians in the remaining
$d-s=18$ dimensions.  We sample $n=1000$ points, and correctly
recover the multimodal features using $\alpha=0.1$.  The results of
the subsequent mean shift clustering using $h=0.06$ are shown in
Figure \ref{fig::example}, along with contours of the true density.

\section{Conclusion}
\vspace{-0.1in}
We have proposed a new method for
feature selection in high-dimensinal clustering problems.
We have given bounds on the error rate
in terms of clustering loss and Hausdorff distance.
In future work, we will address the following issues:
\begin{enum}
\item The marginal signature assumption (A4) is quite strong.
We do not know of any feature selection method for clustering
that can succeed without some assumption like this.
Either relaxing the assumption or proving that it is necessary
is a top priority.
\item The bounds on clustering loss can probably be improved.
This involves a careful study of the properties of the flow near cluster boundaries.
\item We conjecture that the Hausdorff bound is minimax.
We think this can be proved using techniques like those in
Romano (1988).
\end{enum}




\section*{Acknowledgements}
This research is supported in part by NSF awards IIS-1116458 and CAREER IIS-1252412.


\subsubsection*{References}

\small{
[1] Arias-Castro, Mason, Pelletier (2013).
On the estimation of the gradient lines of a density
and the consistency of the mean-shift algorithm.
Manuscript.

[2] Audibert, Jean-Yves, and Alexandre B. Tsybakov. (2007).
Fast learning rates for plug-in classifiers. 
{\em The Annals of Statistics}, 35, 608-633.

[3] Chacon, J. (2012).
Clusters and water flows: a novel approach to modal clustering through Morse theory.
arxiv:1212.1384.

[4] Chan, Yao-ban, and Peter Hall. (2010).
Using evidence of mixed populations to select variables for clustering very high-dimensional data. 
{\em Journal of the American Statistical Association}, 105, 798-809.

[5] Cheng, Yizong. (1995).
Mean shift, mode seeking, and clustering. 
{\em IEEE Transactions on Pattern Analysis and Machine Intelligence},  
17, 790-799.

[6] Comaniciu, Dorin, and Peter Meer. (2002).
Mean shift: A robust approach toward feature space analysis. 
{\em IEEE Transactions onPattern Analysis and Machine Intelligence}, 24, 603-619.

[7] Einbeck, Jochen. (2011).
Bandwidth selection for mean-shift based unsupervised learning techniques: a unified approach via self-coverage.
Journal of pattern recognition research. 6, 175-192.

[8] Fan, Jianqing, and Jinchi Lv. (2008).
Sure independence screening for ultrahigh dimensional feature space. 
{\em Journal of the Royal Statistical Society: Series B}, 70, 849-911.

[9] Guo, Jian, et al. (2010).
Pairwise Variable Selection for 
High-Dimensional Model-Based Clustering. 
{\em Biometrics}, 66, 793-804.

[10] Hartigan, John A., and P. M. Hartigan. (1985).
The dip test of unimodality. 
{\em The Annals of Statistics}, 13, 70-84.

[11] Muller, Dietrich Werner, and Gunther Sawitzki. (1991).
Excess mass estimates and tests for multimodality. 
{\em Journal of the American Statistical Association},  86, 738-746.

[12] Pan, W. and Shen, X. (2007).
Penalized model-based clustering with application to variable selection.
{\em The Journal of Machine Learning Research}, 8, 1145-1164.

[13] Raftery, Adrian E., and Nema Dean. (2006).
Variable selection for model-based clustering. 
{\em Journal of the American Statistical Association}, 101, 168-178.

[14] Romano, J. (1988).
On weak convergence and optimality of kernel density estimates of the mode. 
{\em The Annals of Statistics} 16, 629-647.

[15] Sriperumbudur, Bharath K., et al. (2009).
Kernel Choice and Classifiability for RKHS Embeddings of Probability Distributions. NIPS.

[16] Sun, W., Wang, J. and Fang, Y. (2012).
Regularized k-means clustering of high-dimensional data and its asymptotic consistency.
{\em Electronic Journal of Statistics}, 6, 148-167.

[17] Wand, M. P., Duong, T. and Chacon, J. (2011). 
Asymptotics for general multivariate kernel density derivative estimators. 
{\em Statistica Sinica}, 21, 807-840.

[18] Witten, Daniela M., and Robert Tibshirani. (2010).
A framework for feature selection in clustering. 
{\em Journal of the American Statistical Association},  105, 713-726.

\newpage

\begin{center}
{\bf Appendix}
\end{center}

{\bf Proof of Lemma \ref{lemma::prop1}.}
Let $g$ and $H$ be the gradient and Hessian of $p$ and let
$\tilde g$ and $\tilde H$ be the gradient and Hessian of $\tilde p$.
Let $m$ be a mode of $p$ and let
$B = B(m,\epsilon)$ be a closed ball around $m$ where $\epsilon = \sqrt{8\eta_0}$.
The ball excludes any other mode of $p$ since
$\sqrt{8\eta_0} < a$.
Expanding $p$ at $x\in B$ we have
\begin{equation}\label{eq::Taylor}
p(x) = p(m) + \frac{1}{2} (x-m)^T H(m) (x-m) + R(x)
\end{equation}
where
$|R(x)| \leq \kappa_3 ||x-m||^3/6$.

Since $\tilde p$ is bounded and continuous,
it has at least one maximizer $\tilde m$ over $B$.
We now show that $\tilde m$ must be in the interior of $B$.
Let $0 < \alpha < \beta < 1$ and 
write
$B = A_0 \bigcup A_1 \bigcup A_2$ where
$A_0 = \{x:\ ||x-m|| \leq \alpha\epsilon\}$,
$A_1 = \{x:\ \alpha\epsilon < ||x-m|| \leq \beta\epsilon\}$,
$A_2 = \{x:\ \beta\epsilon < ||x-m|| \leq \epsilon\}$.
For any $x\in A_0$, by (\ref{eq::Taylor}),
$$
\tilde p(x) \geq 
p(x) -\eta_0 \geq
-\frac{B}{2}\alpha^2 \epsilon^2 - \frac{\kappa_3 \alpha^3\epsilon^3}{6} - \eta_0.
$$
For any $x\in A_2$,
$$
\tilde p(x) \leq 
p(x) +\eta_0 \leq
-\frac{b}{2}\beta^2 \epsilon^2 + \frac{\kappa_3 \beta^3\epsilon^3}{6} + \eta_0.
$$
Then if
\begin{equation}\label{eq::alphabeta}
\frac{\epsilon^2}{2}[b \beta^2 - B \alpha^2] -
\frac{\kappa_3 \epsilon^3 (\alpha^3 + \beta^3)}{6} > 2\eta_0
\end{equation}
we will be able to conclude that
$$
\inf_{x\in A_0}\tilde p(x) > \sup_{x\in A_2}\tilde p(x).
$$
Choose $\alpha$ and $\beta$ to satisfy
$(\alpha/\beta) = \sqrt{b/B}$ and
$\kappa_3 \epsilon(\alpha^3+\beta^3)/6 < 1/4$.
It follows that
(\ref{eq::alphabeta}) holds
and so
$\inf_{x\in A_0}\tilde p(x) > \sup_{x\in A_2}\tilde p(x)$.
Hence,
any maximizer of $\tilde p$ in $B$
is in $A_0$ and hence is interior to $B$.
It follows that
$\tilde g(\tilde m) = (0,\ldots, 0)^T$.
Also,
$$
\lambda_1(\tilde H(\tilde m)) \leq
\lambda_1(H(\tilde m)) \leq - b + \eta_2 < - b/2
$$
since $\eta_2 < b/2$.
Hence,
$\tilde p$
has a local mode $\tilde m$ in the interior of $B$ with zero gradient and
negative definite Hessian.

Now we show that
$\tilde m$ is unique.
Suppose $\tilde p$ has two modes
$x$ and $y$ in the interior of $B$.
Recall that the exact Taylor expansion of a vector-valued function $f$ is
$f(a+t) = f(a)+ t^T \int_0^1 D f(a+ut)du$.
So,
$$
(0,\ldots, 0)^T =
\tilde g(x) - \tilde g(y)=
(y-x)^T\int_0^1 \tilde H(x + u(y-x)) du.
$$
Multiple both sides by $y-x$ and conclude that
\begin{align*}
0 &=
\int_0^1 (y-x)^T \tilde H(x + u(y-x)) (y-x) du \leq
||y-x||\sup_u \lambda_1(\tilde H(x + u(y-x)))\\
& \leq
||y-x||\sup_u [\lambda_1(H(x + u(y-x))) + \eta_2]\\
& \leq
||y-x||\ [- b + \eta_2] < - \frac{ b ||y-x||}{2}
\end{align*}
which is a contradiction.

Now we show that
$\tilde p$ has no other modes.
Let
$B_j = B(m,\epsilon_j)$ 
and suppose that
$\tilde p$ has a local mode at
$x\in \left(\bigcup_{j=1}^k B(x_j,\epsilon)\right)^c$.
By a symmetric argument to the one above,
$p$ also must have a local mode
in $B(x,\epsilon)$.
This contradicts the fact that
$m_1,\ldots, m_k$ are the unique modes of $p$. $\Box$

\vspace{1cm}

{\bf Proof Outline for Lemma \ref{lemma::boundary}.}
By assumption, $p(x)$ can be approximated by a quadratic
in a neighborhood of $m$.
Specifically,
we have that
$p(x) = p(m) - (1/2) (x-m)^T H (x-m) + R$
where $H = H(m)$ and
$|R| \leq ||x-m||^3 \kappa_3/6$.
There exists $c_1$ such that,
if
$||x-m|| \leq c_1$
then
$||x-m||^3 \kappa_3/6$ is much smaller than
$B ||x-m||^2/2$ and hence
the quadratic approximation 
$p(x) \approx p(m) - (1/2) (x-m)^T H (x-m)$
is accurate.

Case 1: $||x_0-m|| \leq c_1$.
In this case, the proof of Lemma 5 of Arias-Castro et al shows that
$\pi(t) - m = e^{tH}(x_0-m) + \xi$
where
$\xi = O(||x-m||^3 \kappa_3/6)$ and so
$\pi(t) - m \approx e^{tH}(x_0-m)$.
In particular,
$\pi(t_\epsilon) - m \approx e^{t_\epsilon H}(x_0-m)$ and thus
$$
\sqrt{\epsilon} \approx ||e^{t_\epsilon H}||\ ||x_0-m|| \leq e^{-b t_\epsilon} ||x_0-m||
$$
so that
$e^{-b t_\epsilon} \geq \frac{\sqrt{\epsilon}}{||x_0-m||}$
and so
$$
t_\epsilon \leq
\frac{\frac{1}{2}\log(1/\epsilon) + \log ||x_0-m||}{b} \leq
\frac{C_6}{||g(x_0)||} +
\frac{\frac{1}{2}\log(1/\epsilon) + \log ||x_0-m||}{b}.
$$

Case 2: $||x_0-m|| > c_1$.
In this case, the starting point $x_0$
is not in the quadratic zone.
There exists $t_1<\infty$ 
(not depending on $\epsilon$)
such that
$||\pi(t_1)-m||\leq c_1$.
Let us first bound $t_1$.
Since 
$\pi'(t) = g(\pi(t))$
we have
$\pi(t) = \int_0^t g(\pi(s)) ds + x_0$ and thus
$\pi(t_1) - x_0 = \int_0^{t_1} g(\pi(s)) ds.$
Now
$g(x) = g(x_0) + H_s(x-x_0)$ 
where $H_s$
is the Hessian evaluated at some point between $x_0$ and $\pi(s)$.
Thus,
$\pi(t_1) - x_0 = t_1 g(x_0) + 
\int_0^{t_1} H_s(x(s)-x_0) ds$
and therefore
$t_1 g(x_0) = 
\pi(t_1) - x_0 - \int_0^{t_1} H_s(\pi(s)-x_0) ds.$
It follows that
$$
t_1 ||g(x_0)|| \leq
||\pi(t_1) - x_0|| + \int_0^{t_1} ||H(\pi(s)-x_0)|| ds \leq
||m - x_0|| + \int_0^{t_1} ||H(\pi(s)-x_0)|| ds 
$$
and
$$
t_1 \leq \frac{||m - x_0|| + \int_0^{t_1} ||H(\pi(s)-x_0)|| ds }{||g(x_0)||} \equiv
\frac{C_6}{||g(x_0)||}.
$$
Now consider the flow $\tilde \pi$ starting at $x_1$.
This is the same as the original flow except starting at $x_1$ rather than $x_0$.
There exists $\tilde t_\epsilon$ on this flow such that
$t_\epsilon = t_1 + \tilde t_\epsilon$.
Applying case 1,
$$
\tilde t_\epsilon \leq
\frac{\frac{1}{2}\log(1/\epsilon) + \log ||x_1-m||}{b} \leq
\frac{\frac{1}{2}\log(1/\epsilon) + \log ||x_0-m||}{b}.
$$
Thus,
$$
t_\epsilon = 
t_1 + \tilde t_\epsilon \leq
\frac{C_6}{||g(x_0)||} +
\frac{\frac{1}{2}\log(1/\epsilon) + \log ||x_0-m||}{b}.\ \ \ \Box
$$

\end{document}